\documentclass{amsart}
\usepackage{amssymb, amsmath}
\usepackage[a4paper,top=2cm,bottom=2cm,left=1.8cm,right=1.8cm]{geometry}
\usepackage{array}
\usepackage{xcolor}
\newtheorem{Theorem}{Theorem}
\newtheorem{Definition}[Theorem]{Definition}
\newtheorem{Remark}[Theorem]{Remark}

\newtheorem{Lemma}[Theorem]{Lemma}
\newtheorem{Corollary}[Theorem]{Corollary}

\newcommand{\spa}{\mbox{span}}
\newcommand{\F}{F\langle X\rangle}

\newcommand{\Ft}{F\langle X, \mbox{Tr}\rangle}
\newcommand{\V}{\mbox{var}}
\DeclareMathOperator{\Id}{Id}

\newcommand{\tr}{\mbox{tr}}
\newcommand{\Tr}{\mbox{Tr}}

\begin{document}
\title[Matrix algebras with degenerate traces and trace identities]{Matrix algebras with degenerate traces and trace identities}

\author{Antonio Ioppolo}
\address{IMECC, UNICAMP, S\'ergio Buarque de Holanda 651, 13083-859 Campinas, SP, Brazil}
\email{ioppolo@unicamp.br}

\thanks{A. Ioppolo was supported by the Fapesp post-doctoral grant number 2018/17464-3}

\author{Plamen Koshlukov}
\address{IMECC, UNICAMP, S\'ergio Buarque de Holanda 651, 13083-859 Campinas, SP, Brazil}
\email{plamen@unicamp.br}

\thanks{P. Koshlukov was partially supported by CNPq grant number 302238/2019-0 and by FAPESP grant number 2018/23690-6}

\author{Daniela La Mattina}
\address{Dipartimento di Matematica e Informatica, Universit\`a degli Studi di Palermo, Via Archirafi 34, 90123, Palermo, Italy}
\email{daniela.lamattina@unipa.it}

\thanks{D. La Mattina was partially supported by GNSAGA-INDAM}



\begin{abstract}
In this paper we study matrix algebras with a degenerate trace in the framework of the theory of polynomial identities.
The first part is devoted to the study of the algebra $D_n$ of $n \times n$ diagonal matrices. 
We prove that, in case of a degenerate trace, all its trace identities follow by the commutativity law and by pure trace identities.  Moreover  we relate the trace identities of $D_{n+1}$ endowed with a degenerate trace, to those  of $D_n$ with the corresponding trace. This allows us to determine the generators of the trace T-ideal of $D_3$.

In the second part we study commutative subalgebras of $M_k(F)$, denoted by $C_k$ of the type $F + J$ that can be endowed with the so-called strange traces: $\tr(a+j) = \alpha a + \beta j$, for any $a+j \in C_k$, $\alpha$, $\beta \in F$. Here $J$ is the radical of $C_k$. In case $\beta = 0$ such a trace is degenerate, and we study the trace identities satisfied by the algebra $C_k$, for every $k \geq 2$. Moreover we  prove that these algebras generate the so-called minimal varieties of polynomial growth. 

In the last part of the paper,  devoted to the study of varieties of polynomial growth, we completely classify the subvarieties of the varieties of algebras  of almost polynomial growth introduced in (\cite{IoppoloKoshlukovLaMattina2021}).
\end{abstract}

\maketitle

\section{Introduction}

All algebras and vector spaces we consider will be associative and over a fixed field $F$ of characteristic zero.
The aim of this paper is to present several results of the theory of polynomial identities in the setting of algebras with trace. The theory of trace identities, which is strictly related to the invariant theory of $n \times n$ matrices, represents an important area of the modern mathematics. The main contributions to this field  are given by the results of Procesi and Razmyslov obtained independently in \cite{Procesi1976, Razmyslov1974}. Here we want to highlight how the theory of trace identities contributed decisively to the development of many areas of the theory of PI-algebras (algebras satisfying at least one non-trivial polynomial identity).

In $1972$ Regev introduced the sequence of codimensions $c_n(A)$, $n=1, 2, \dots$, of an associative algebra $A$ (\cite{Regev1972}).  Recall
that if $P_n$ is the vector space of multilinear polynomials in the
non-commutative variables $x_1$, \dots, $x_n$ and $\Id(A)$ is the T-ideal  of identities of $A$ then $c_n(A)=\dim P_n/(P_n\cap \Id(A))$. The codimension sequence is one of the most important numerical invariants of an ideal of identities. It is well known that when the base field is of characteristic 0, every ideal of identities $I=\Id(A)$ is generated by its multilinear elements, that is by the intersections $P_n\cap I$, $n\ge 1$. But $P_n$ is a left module over the symmetric group $S_n$ and it is isomorphic to the left regular module $FS_n$. Ideals of identities are closed under permutations of the variables hence $P_n\cap I$ is a submodule and, so, $P_n(A)=P_n/(P_n\cap I)$ is an $S_n$-module as well. In characteristic 0 one can employ the well developed theory of representations of $S_n$ in order to study polynomial identities, and this constitutes one of the principal methods of studying PI-algebras. 

The precise knowledge of the identities satisfied by an algebra is an extremely hard task; a fruitful approach to obtain some information about them is through the study of the corresponding codimensions. 
Nevertheless computing explicitly the codimension sequence of an algebra is also very difficult. Indeed there are only few examples of algebras whose codimension sequence is known explicitly.  It is more feasible to study the asympotics of such a sequence. The key result in this area says that  the codimension sequence of a PI-algebra is exponentially bounded (\cite{Regev1972}) and its exponential growth is an integer (see \cite{GZbook}). More precisely, Regev in \cite{Regev1972} proved that if $A$ satisfies an identity of degree $d$ then $c_n(A)\le (d-1)^{2n}$ for every $n$. This comes to justify why studying $P_n(A)$ should be ``easier'' than $P_n\cap \Id(A)$: the exponential function grows much slower than $n!$, thus in a sense, for $n$ large enough, almost ``all'' of $P_n$ will be identities for $A$. The theorem of Giambruno and Zaicev, see their monograph \cite{GZbook}, states that in characteristic 0, if $A$ is a PI-algebra then the limit $\lim_{n\to\infty} c_n(A)^{1/n}$ always exists and is an integer. This limit is called the PI-exponent of $A$, and is denoted by $\exp(A)$. A similar result holds for large classes of non-associative algebras, and also for algebras with an additional structure: group graded algebras, algebras with involution. Giambruno and Zaicev's theorem became the starting point of an extensive research towards classification of the PI-algebras according to the growth of their codimensions.

In case $A$ is an algebra with trace, as in  the ordinary case, one defines trace  identities, multilinear trace polynomials and trace codimension sequence. The asymptotic behaviour of the trace codimensions of the matrix algebra was studied by Regev: in \cite{Regev1984} he proved that the ordinary and the trace codimensions of the full matrix algebra are asymptotically equal. 

Recently Berele proved in \cite{Berele2020} that the sequence of trace codimensions of a PI-algebra is exponentially bounded if and only if the algebra satisfies  a particular trace identity which he called an inter-trace permuting identity. 

In \cite{IoppoloKoshlukovLaMattina2021} the authors gave a characterization of  the varieties of algebras with trace of polynomial growth, that is varieties  generated by algebras whose sequence of codimensions is polynomially bounded. The varieties of algebras whose codimensions grow like polynomial functions have been object of extensive research, the interested reader can find a wealth of results in that direction in \cite[Chapter 7]{GZbook} and the references therein.

In this paper we focus our attention on matrix algebras. The algebra $M_n(F)$ of $n \times n$ matrices over the field $F$ represents the most well-known and widely used example of an algebra endowed with a trace function. The usual trace of a matrix is defined as the sum of all its elements on the main diagonal. It is known and easy to show that every trace on $M_n(F)$ is proportional to the usual one and it is always non-degenerate. 

In their celebrated theorem, Procesi and Razmyslov showed that all the trace identities of the full matrix algebra are consequences of just one element: the Cayley-Hamilton polynomial. We notice that such a result has no analogue in the ordinary case: while generators of the ideals of identities for $M_n(F)$, $n\le 2$, are known, even for $3 \times 3$ matrices we have no idea what the generators of the ideal of identities look like. In fact, the theorem of Procesi and Razmyslov is one of the most general results in PI-theory.

Now let us consider the subalgebra of $M_n(F)$ of diagonal matrices, denoted as $D_n$. In sharp contrast with the situation for full matrix algebras, there are very many traces that one can define on $D_n$. In case $D_n$ is endowed with the usual trace, its identities have been completely described by Berele in \cite{Berele1996}. In fact, the commutativity of such an algebra implies that a trace on it is just a linear map $D_n \to F$. As  $D_n\cong F^n$ then clearly the dual of the $n$-dimensional vector space $F^n$ gives us all traces on $D_n$. 

The first part of this paper is devoted to the study of diagonal matrix algebras endowed with a degenerate trace.  Our first result shows that, in order to study the trace identities of such  algebras, it is sufficient to consider only pure trace polynomials, that is polynomials in which all variables appear inside traces. 
Then we prove that the trace identities of  $D_{n+1}$  endowed with a degenerate trace, are very closely related to those of $D_n$ with the  trace which is in a sense the restriction of the trace of $D_{n+1}$ to $D_n$. As a consequence,  starting from the trace identities of $D_2$, determined in \cite{IoppoloKoshlukovLaMattina2021,  IoppoloKoshlukovLaMattina2021proc},  we  find the generators of the trace $T$-ideal of the identities of  $D_3$ endowed with all possible degenerate traces. Along the way we also compute the trace codimension sequence of these algebras. It turns out they are expressed in terms of  Stirling numbers of the second kind. 

In the second part of the paper we study certain commutative subalgebras of $M_k(F)$, denoted by $C_k$, of the type $F + J$ where $J$ is the radical of $C_k$. These subalgebras can be viewed as the quotient algebras $F[x]/(x^k)$ in an obvious manner. For every $\alpha$, $\beta \in F$, we define the trace functions $t_{\alpha, \beta}\colon C_k \to F$ as $t_{\alpha, \beta}(a+b)= \alpha a + \beta b$, for every $a\in F$, $b \in J$. Such traces have been called strange traces by Procesi  (\cite{Procesi2008}).
Here we compute explicitly the generators of the trace T-ideal of $C_2$ endowed with the trace $t_{\alpha, \beta}$, for each choice of $\alpha$,  $\beta$. In case $\beta=0$ a complete list of generators is given also for the ideal of trace identities of $C_k$ for any $k$.
In case of degenerate trace, we prove that the algebras $C_k$ generate minimal varieties of polynomial growth, that is varieties  of polynomial growth with the property that every proper subvariety has a smaller degree of its polynomial growth.

The last part of the paper is devoted to the study of varieties of polynomial growth.  We completely classify the subvarieties generated by unitary finite dimensional algebras of the varieties of algebras  of almost polynomial growth (\cite{IoppoloKoshlukovLaMattina2021}). Recall that these are varieties whose codimensions grow exponentially but for each proper subvariety they grow polynomially.

\section{Preliminaries}

Throughout this paper $F$ will denote a field of characteristic zero and $A$ a unitary associative $F$-algebra with trace $\tr$. We say that $A$ is an algebra with trace if it is endowed with a linear map $\tr\colon A \to F$ such that for all $a$, $b \in A$ one has
\[
\tr(ab) = \tr(ba).
\]
In what follows, we shall identify, when it causes no misunderstanding, the element $\alpha \in F$ with $\alpha \cdot 1$ where 1 is the unit of the algebra, that is we assume $F=F\cdot 1\subseteq A$. 

Accordingly, one can construct $\Ft$, the free algebra with trace on the countable set $X = \{ x_1, x_2, \ldots \}$ of free generators where $\Tr$ is a formal trace. Let $\mathcal{M}$ denote the set of all monomials in the elements of $X$. Then $\Ft$ is the algebra generated by the free algebra $\F$ together with the set of central (commuting) variables $\Tr(M)$, $M \in \mathcal{M}$, subject to the conditions that $\Tr(MN) = \Tr(NM)$,
and $\Tr(\Tr(M)N)=\Tr(M)\Tr(N)$, for all $M$, $N \in \mathcal{M}$. In other words, 
\[
\Ft \cong \F \otimes F[\Tr(M) \mid M \in \mathcal{M}].
\]
The elements of the free algebra with trace are called trace polynomials. 

A trace polynomial $f(x_1, \ldots, x_n, \Tr) \in \Ft$ is a trace identity for $A$ if, after substituting the variables $x_i$ with arbitrary elements $a_i \in A$ and $\Tr$ with the trace $\tr$, we obtain 0. We denote by $\Id^{tr}(A)$ the set of trace identities of $A$, which is a trace $T$-ideal ($T^{tr}$-ideal) of the free algebra with trace. In other words it is an ideal which is invariant under all endomorphisms of $\Ft$.

As in the ordinary case, $\Id^{tr}(A)$ is completely determined by the  multilinear polynomials it contains.

\begin{Definition}
The vector space of multilinear elements of the free algebra with trace in the first $n$ variables is called the space of multilinear trace polynomials in $x_1$, \dots, $x_n$. It is denoted by $MT_n$ ($MT$ comes from \textsl{mixed trace}). Its elements are linear combinations of expressions of the type
\[
\Tr(x_{i_1} \cdots x_{i_a}) \cdots \Tr(x_{j_1} \cdots x_{j_b}) x_{l_1} \cdots x_{l_c}
\]
where $	\left \{ i_1, \ldots, i_a, \ldots, j_1, \ldots, j_b, l_1, \ldots, l_c \right \} = \left \{ 1, \ldots, n \right \} $.
\end{Definition}

The non-negative integer
\[
c_n^{tr}(A) = \dim_F \dfrac{ MT_n}{MT_n \cap \Id^{tr}(A)} 
\]
is called the $n$-th trace codimension of $A$. It is bounded from above by  $(n+1)!$, that is  the dimension of the vector space  $MT_n$.

A prominent role among the elements of $MT_n$ is played by the so-called pure trace polynomials, i.e., polynomials such that all the variables $x_1$, \dots, $x_n$ appear inside traces.

\begin{Definition}
The vector space of multilinear pure trace polynomials in $x_1$, \dots, $x_n$ is spanned by the elements
\[
PT_n = \spa_F 	 \{ \Tr(x_{i_1} \cdots x_{i_a}) \cdots \Tr(x_{j_1} \ldots x_{j_b}) \mid  \{ i_1, \ldots, j_b  \} =  \{ 1, \ldots, n  \}  \}. 
\]
\end{Definition}

\section{Traces on matrix algebras}

In this section we collect some basic facts and results concerning matrix algebras endowed with a trace function. 

The algebra $M_n(F)$ of the $ n \times n $ matrices over $F$ represents one of the most prominent examples of an algebra with trace. The usual trace on such an algebra, we denote it as $t_1$, is defined for every $a \in M_n(F)$ as 
\[
t_1(a) = t_1  \begin{pmatrix} a_{11} & \cdots & a_{1n} \\ \vdots &
\ddots & \vdots \\ a_{n1} & \cdots &
a_{nn} \end{pmatrix}  = a_{11} + \cdots + a_{nn} \in F.
\]
When necessary we shall consider the trace assuming its values in the algebra by identifying $F$ and $FI_n$ where $I_n$ is the identity matrix (or more generally the unit of our algebra).
The following fact is well known and an easy exercise in elementary linear algebra. It shows that every trace on $M_n(F)$ is proportional to the usual one. 

\begin{Lemma} \label{traces on matrices}
Let $f\colon M_n(F) \to F$ be a trace. Then there exists $\alpha \in F$ such that $f = \alpha t_1$.
\end{Lemma}

Such traces on $M_n(F)$ are not very ``interesting'' from the PI-point of view: they can be obtained from the usual trace by a ``rescaling''. 

In sharp contrast with the above result, there are very many traces on the algebra $D_n = D_n(F)$ of $n \times n$ diagonal matrices over $F$. Indeed, since $D_n$ is commutative, every trace on it is just a linear map $D_n \rightarrow F$. The following remark illustrates this situation. 

\begin{Remark}\label{traces_on_Dn}
If $\tr$ is a trace on $D_n$ then there exist $\alpha_1$, \dots, $\alpha_n\in F$ such that, for any $a=\mbox{diag}(a_{11},\ldots, a_{nn})\in D_n$,
\[
\tr(a) = \alpha_1 a_{11}+\cdots+ \alpha_n a_{nn}.
\]
This means that if $D_n^*$ is the dual of the vector space $D_n$ then the traces on $D_n$ are in 1--1 correspondence with the elements of $D_n^*$.
\end{Remark}
We shall denote such a trace by $t_{\alpha_1, \ldots, \alpha_n}$. Furthermore  $D_n^{t_{\alpha_1, \ldots, \alpha_n}}$ will indicate the algebra $D_n$ with the trace $t_{\alpha_1, \ldots, \alpha_n}$. 

Let $(A, t)$ and $(B, t')$ be two algebras with trace. An isomorphism of algebras $\varphi\colon A \to B$ is said to be an isomorphism of algebras with trace if $\varphi(t(a)) = t'(\varphi(a))$ for every $a \in A$. In this case we write $A \cong_{tr} B$.

\begin{Remark} \label{isomorphic Dn}
Let $S_n$ be the symmetric group of order $n$ on the set $\{ 1,2, \ldots, n \}$. For every $\sigma \in S_n$ we have that
\[
D_n^{t_{\alpha_1, \ldots, \alpha_n}} \cong_{tr} D_n^{t_{\alpha_{\sigma(1)}, \ldots, \alpha_{\sigma(n)}}}.
\]
\end{Remark}
\begin{proof}
Let $e_{ii}$ be the diagonal matrix units, $i = 1$, \dots, $n$. The linear map $\varphi \colon D_n^{t_{\alpha_1, \ldots, \alpha_n}} \to D_n^{t_{\alpha_{\sigma(1)}, \ldots, \alpha_{\sigma(n)}}}$ is defined by $ \varphi(e_{ii}) = e_{\sigma(i) \sigma(i)}$, for all $i = 1$, \dots, $n$. It is an isomorphism of algebras, and, moreover, for every $i = 1$, \dots, $n$, we have that
\[
\varphi (t_{\alpha_1, \ldots, \alpha_n}(e_{ii})) = \varphi (\alpha_i) \equiv \alpha_i (e_{11} + \cdots + e_{nn}) = t_{\alpha_{\sigma(1)}, \ldots, \alpha_{\sigma(n)}} (e_{\sigma(i) \sigma(i)}) = t_{\alpha_{\sigma(1)}, \ldots, \alpha_{\sigma(n)}} ( \varphi (e_{ii}) )
\]
and the proof is complete.
\end{proof}

A trace $\tr$ on an algebra $A$ is degenerate if there exists a non-zero element $a \in A$ such that, for every $b\in A$,
\[
\tr(ab) = 0.
\]
This means that the symmetric bilinear form on $A$ given by $\frak{b}(a,b) = \tr(ab)$ is degenerate. 

In the following lemma we describe degenerate traces on $D_n$. 

\begin{Lemma}\label{degenerate traces}
Let $D_n^{t_{\alpha_1, \ldots, \alpha_n}}$ be the algebra of $n \times n$ diagonal matrices endowed with the trace  
$t_{\alpha_1, \ldots, \alpha_n}$. Such a trace is degenerate if and only if $\alpha_i = 0$, for some $i = 1$, \dots, $n$. 
\end{Lemma}
\begin{proof}
Let $t_{\alpha_1, \ldots, \alpha_n}$ be degenerate. By definition, there exists a non-zero element $a = diag(a_{11}, \ldots, a_{nn}) \in D_n$ such that, for every $b\in D_n$, $\tr(ab) = 0$. Since $a$ is non-zero, then $a_{ii} \neq 0 $, for some $i = 1$, \dots, $n$. Hence
\[
\tr(a e_{ii}) = \tr (a_{ii} e_{ii}) = a_{ii} \tr (e_{ii}) = \alpha_i a_{ii} = 0.
\]
It follows that $\alpha_i = 0$ and we have proved the first implication. 

In order to prove the opposite implication, fix that
$\alpha_i = 0$, and let us consider the element $a = e_{ii}$. It is immediate that $\tr (ab) = 0$ for every $b \in D_n$, and therefore the trace is degenerate.
\end{proof}

\section{Diagonal matrices with a degenerate trace}

In this section we focus our attention on the diagonal matrix algebra $D_{n+1}$ endowed with a degenerate trace $t_{\alpha_1, \ldots, \alpha_{n+1}}$. Since the trace is degenerate, by Lemma \ref{degenerate traces}, at least one of the $\alpha_i$'s is zero. In light of Remark \ref{isomorphic Dn}, we may suppose that $\alpha_{n+1} = 0$ and so the trace on  $D_{n+1}$ is $t_{\alpha_1, \ldots, \alpha_{n}, 0}$.

\begin{Theorem} \label{Ninni's conjecture}
Every trace identity of $D_{n+1}^{t_{\alpha_1, \ldots, \alpha_{n}, 0}}$, which is not a consequence of the commutator $[x_1, x_2] \equiv 0$, is a consequence of pure trace identities.
\end{Theorem}
\begin{proof}
Let $f(x_1, \ldots, x_k)$ be a multilinear trace identity of $D_{n+1}^{t_{\alpha_1, \ldots, \alpha_{n}, 0}}$ of degree $k$ which is not a consequence of the commutator $[x_1, x_2] \equiv 0$. 

If $f$ is a pure trace polynomial then there is nothing to prove. If not, there exists at least one variable, say $x_k$, appearing outside the traces of one or more monomials of $f$. We can write $f$ as:
\[
f(x_1, \ldots, x_k) = g(x_1, \ldots, x_{k-1}) x_k + h(x_1, \ldots, x_k),
\]
where
\begin{itemize}
\item[•] $g(x_1, \ldots, x_{k-1})$ is a trace polynomial (not necessarily pure trace),
\item[•] $h(x_1, \ldots, x_k)$ is a trace polynomial in which $x_k$ appears always inside some traces.
\end{itemize}

Let us consider the evaluation $x_k = e_{n+1,n+1}$. It is clear that, for every $a_1$, \dots, $a_{k-1} \in D_{n+1}$, we have 
\[
h(a_1, \ldots, a_{k-1}, e_{n+1,n+1}) = 0.
\]
Now we focus our attention on the polynomial $g(x_1, \ldots, x_{k-1})$. Let us consider the monomials of $g$ having the largest number of variables outside the traces, and let us denote such a number by $l$. Clearly there might be several of these monomials. Fix one of them, say $M_1$, and let $r+l = k-1$, then
\[
M_1 = m_1(x_{i_1}, \ldots, x_{i_r}) x_{j_1} \cdots  x_{j_l}.
\]
We consider the evaluation (substitution):
\[
x_{j_1} = \cdots = x_{j_l} = e_{n+1,n+1}.
\]
Clearly all the monomials of $g$ having less than $l$ variables outside the traces vanish under such a substitution. The same happens to all the monomials with exactly $l$ variables outside the traces but in which at least one of these $l$ variables is not among $\{ x_{j_1}, \ldots,  x_{j_l} \}$. We are left just with monomials having exactly the variables $x_{j_1}$, \dots,  $x_{j_l}$ outside the traces and all remaining variables inside traces. These monomials can differ from $M_1$ just in their pure trace part. In effect, we have
\[
g_{\mid x_{j_1} = \cdots = x_{j_l} = e_{n+1,n+1}} = \underbrace{(m_1(x_{i_1}, \ldots, x_{i_r}) + \cdots + m_s(x_{i_1}, \ldots, x_{i_r})) }_{m(x_{i_1}, \ldots, x_{i_r})} e_{n+1, n+1}. 
\]

CLAIM: $m(x_{i_1}, \ldots, x_{i_{r}})$ is a pure trace identity of $D_{n+1}^{t_{\alpha_1, \ldots, \alpha_n, 0}}$. 

If not, there exist diagonal matrices $d_1$, \dots, $d_{r} \in D_{n+1} $ such that $ m(d_1, \ldots, d_{r}) \neq 0$. But since $m$ is a pure trace polynomial, such a non-zero evaluation is, for some $\mu \in F$, equal to $\mu (e_{11}+ \cdots + e_{n+1,n+1})$. Hence we get a contradiction, since we have
\[
f|_{x_k = x_{j_1} = \cdots = x_{j_l} = e_{n+1,n+1}, x_{i_1} = d_1, \ldots, x_{i_r} = d_r} = \mu e_{n+1,n+1} \neq 0. 
\]

Using the same approach we can deal with all remaining monomials having exactly $l$ variables outside the traces. In any case, we obtain that some part of the polynomial $g$ is a consequence of a pure trace identity. 

Then we consider the monomials of $g$ having exactly $l-1$ variables outside the traces (if there are not such monomials we consider $l-2$, $l-3$, and so on). We apply the same technique. As this process cannot go on infinitely many steps, at the end of it we obtain that $g(x_1, \ldots, x_{k-1})$ is a consequence of pure trace identities. 

It remains to deal with the polynomial $h(x_1, \ldots, x_k)$. We use the same approach but taking into account that the variable $x_k$ is always inside a trace. For instance, we can start with the variable $x_{k-1}$ in case it appears outside the trace in at least one monomial of $h$; if not we go to $x_{k-2}$, and so on.

This process will stop, after several steps, and the proof is complete. 
\end{proof}

In the following lemma we shall see that the identities of $D_{n+1}^{t_{\alpha_1, \ldots, \alpha_n, 0}}$ and of $D_{n}^{t_{\alpha_1, \ldots, \alpha_n}}$ are very closely related.

\begin{Lemma} \label{link between identities of Dn+1 and Dn}
Let $\alpha_1$, \dots, $\alpha_n \in F \setminus \{0\}$. Then $f(x_1, \ldots, x_k) \in \Id^{tr}(D_n^{t_{\alpha_1, \ldots, \alpha_n}})$ if and only if $\Tr(f(x_1, \ldots, x_k) x_{k+1}) \in \Id^{tr}(D_{n+1}^{t_{\alpha_1, \ldots, \alpha_n, 0}})$. 
\end{Lemma}
\begin{proof}
Let $f(x_1, \ldots, x_k) \in \Id^{tr}(D_n^{t_{\alpha_1, \ldots, \alpha_n}})$. In order to prove that $\Tr(f(x_1, \ldots, x_k) x_{k+1}) \in \Id^{tr}(D_{n+1}^{t_{\alpha_1, \ldots, \alpha_n, 0}})$ we shall evaluate the variables on the basis $\{ e_{11}, \ldots, e_{n+1, n+1} \}$ of $D_{n+1}^{t_{\alpha_1, \ldots, \alpha_n, 0}}$. 

If no variable among $x_1$, \dots, $x_k$ is evaluated on $e_{n+1,n+1}$, then the result follows since $f \in \Id^{tr}(D_n^{t_{\alpha_1, \ldots, \alpha_n}})$. In fact, for every evaluation $a_1$, \dots, $a_k$ of this kind, we get that $f(a_1, \ldots, a_k) = \mu e_{n+1,n+1}$, for some $\mu \in F.$  Clearly, for every evaluation $a_{k+1}$ of the variable $x_{k+1}$, we get that 
$\Tr(f(x_1, \ldots, x_k) x_{k+1})$ vanishes (recall that $t_{\alpha_1, \ldots, \alpha_n, 0} (e_{n+1, n+1}) = 0$).

Hence we have to evaluate at least one variable, say $x_k$, to $e_{n+1,n+1}$. Now let us consider a monomial $M(x_1, \ldots, x_k)$ of $f$. If $x_k$ is inside a trace in $M$, then the evaluation of such a monomial will be zero and, of course, it will be also equal to the evaluation of $\Tr( M(x_1, \ldots, x_k) x_{k+1} )$. Otherwise $M$ will be a monomial of the type
\[
M(x_1, \ldots, x_k) = \Tr(\cdots) \cdots \Tr(\cdots) \cdots x_k
\]
and so the evaluation of $\Tr( M(x_1, \ldots, x_k) x_{k+1} )$ will be zero also in this case. Since this happens for every monomial of $f$, one of the implications is proved.

Let now $\Tr(f(x_1, \ldots, x_k) x_{k+1}) \in \Id^{tr}(D_{n+1}^{t_{\alpha_1, \ldots, \alpha_n, 0}})$. Suppose, by contradiction, that $f(x_1, \ldots, x_k)$ is not a trace identity of $D_{n}^{t_{\alpha_1, \ldots, \alpha_n}}$. Hence  there exist $a_1$, \dots, $a_k \in D_n^{t_{\alpha_1, \ldots, \alpha_n}}$ such that, for some $\mu_i \in F$,
\[
f(a_1, \ldots, a_k) = \sum_{i = 1}^{n} \mu_i e_{ii} \neq 0.
\]
Since such an evaluation does not vanish, at least one of the $\mu_i$'s is non-zero, say $\mu_1\ne 0$. Now we evaluate $x_{k+1}$ to $e_{11}$ and we reach a contradiction since
\[
\Tr 	\left ( f(a_1, \ldots, a_k) e_{11} \right ) =
\Tr 	\left ( \left ( \sum_{i = 1}^{n} \mu_i e_{ii} \right ) e_{11} \right ) = \Tr(\mu_1 e_{11}) = \alpha_1 \mu_1 \neq 0. 
\]
\end{proof}

The above lemma implies the following corollary. 

\begin{Corollary} \label{f is pure trace, form of Tr(fx)}
If $f(x_1, \ldots, x_k)$ is a pure trace identity of $D_{n}^{t_{\alpha_1, \ldots, \alpha_n}}$  then $f(x_1, \ldots, x_k) \in \Id^{tr} 	\left ( D_{n+1}^{t_{\alpha_1, \ldots, \alpha_n, 0}} \right )$.
\end{Corollary}
\begin{proof}
Since $f$ is a pure trace polynomial, it follows that
$$
\Tr(f(x_1, \ldots, x_k) x_{k+1}) = f(x_1, \ldots, x_k) \Tr(x_{k+1}).
$$
By Lemma \ref{link between identities of Dn+1 and Dn}, we have that $f(x_1, \ldots, x_k) \Tr(x_{k+1})$ is a trace identity on $D_{n+1}^{t_{\alpha_1, \ldots, \alpha_n, 0}}$. Hence it is clear that $f$ has to be a trace identity on $D_{n+1}^{t_{\alpha_1, \ldots, \alpha_n, 0}}$ and the proof is complete. 
\end{proof}

\begin{Lemma} \label{form of the identities of Dn+1}
Every trace identity of $D_{n+1}^{t_{\alpha_1, \ldots, \alpha_n, 0}}$, which is not a consequence of the commutator $[x_1, x_2] \equiv 0$, is a consequence of trace identities of the form
\[
\Tr(f(x_1, \ldots, x_k) x_{k+1}).
\]
\end{Lemma}
\begin{proof}
Let $g$ be a trace identity of $D_{n+1}^{t_{\alpha_1, \ldots, \alpha_n, 0}}$ and assume that $g$ is not a consequence of $[x_1, x_2].$ Then, by Theorem \ref{Ninni's conjecture}, $g$ is a consequence of pure trace identities. Let $f'(x_1, \ldots, x_{k})$ be one of these pure trace identities. We shall construct a new polynomial $f(x_1, \ldots, x_{k-1})$ starting from $f'$ in the following way. Let $M$ be a monomial of $f'$ of the form
\[
M = \Tr(x_{i_1} \cdots x_{i_a}) \cdots \Tr(x_{j_1} \cdots x_{j_b}) \Tr(x_{l_1} \cdots x_{l_{c-1}} x_{l_{c}})
\]
where we assume without loss of generality that $ l_{c} = k$. Then in our new polynomial $f$ we substitute the monomial $M$ by
\[
M' = \Tr(x_{i_1} \cdots x_{i_a}) \cdots \Tr(x_{j_1} \cdots x_{j_b}) x_{l_1} \cdots x_{l_{c-1}}. 
\] 
It follows that $f' = \Tr(f(x_1, \ldots, x_{k-1}) x_{k})$ and the lemma is proved. 
\end{proof}

Now we prove the main result of this section. 

Given trace polynomials $f_1$, \dots, $f_k$ let $\langle f_1, \ldots, f_k \rangle_{T^{tr}}$ denote the trace ideal generated by $f_1$, \dots, $f_k$. Clearly it is the least ideal of trace identities that contains the given polynomials.
\begin{Theorem} \label{identities of Dn+1 from those of Dn}
Let $\alpha_1$, \dots, $\alpha_n \in F \setminus \{0\}$ be non-zero scalars. If 
\[
\Id^{tr} 	\left ( D_n^{t_{\alpha_1, \ldots, \alpha_n}} \right ) = \langle f_1 (x_1, \ldots, x_k), \ldots, f_s (x_1, \ldots, x_h), [x_1, x_2]  \rangle_{T^{tr}}
\] 
then
\[
\Id^{tr} 	\left ( D_{n+1}^{t_{\alpha_1, \ldots, \alpha_n, 0}} \right ) = \langle \Tr \left ( f_1 (x_1, \ldots, x_k) x_{k+1} \right ), \ldots, \Tr \left ( f_s (x_1, \ldots, x_h) x_{h+1} \right ), [x_1, x_2]  \rangle_{T^{tr}}.
\]
\end{Theorem}
\begin{proof}
By Lemma \ref{link between identities of Dn+1 and Dn} we know that $\Tr \left ( f_1 (x_1, \ldots, x_k) x_{k+1} \right )$, \dots, $\Tr \left ( f_s (x_1, \ldots, x_h) x_{h+1} \right )$ are trace identities for $D_{n+1}^{t_{\alpha_1, \ldots, \alpha_n, 0}} $. Moreover the commutator is obviously an identity for both algebras. 

In order to complete the proof, let $f$ be a trace identity of $D_{n+1}^{t_{\alpha_1, \ldots, \alpha_n, 0}}$. If $f$ is a consequence of the commutator $[x_1, x_2] \equiv 0$ then there is nothing to prove. Let us suppose  that $f$ is not a consequence of the commutator. By putting together Lemmas \ref{link between identities of Dn+1 and Dn} and \ref{form of the identities of Dn+1}, it follows that $f$ is a consequence of trace identities of the form
\[
g = \Tr \left ( h (x_1, \ldots, x_k) x_{k+1} \right )
\]
where $h (x_1, \ldots, x_k) \in \Id^{tr}(D_n^{t_{\alpha_1, \ldots, \alpha_n}})$. But then $h (x_1, \ldots, x_k)$ is a consequence of the identities $f_1$, \dots, $f_s$, $[x_1, x_2]$. The proof is now complete. 
\end{proof}

\section{Trace identities on $D_2$}

In this section we focus our attention to the algebra $D_2$ of $2 \times 2$ diagonal matrices over the field $F$. According to Remarks \ref{traces_on_Dn}, \ref{isomorphic Dn}, we can define, up to isomorphism, only the following trace functions on $D_2$:
\begin{enumerate}
\item[1.] 
$t_{\alpha, 0}$, for every $\alpha \in F $, 
\item[2.] 
$t_{\alpha, \alpha}$, for every non-zero $\alpha \in F $, 
\item[3.]
$t_{\alpha, \beta }$, for every distinct non-zero $\alpha$, $\beta \in F $.
\end{enumerate}

The algebras $D_2^{t_{\alpha,0}}$, $D_2^{t_{\alpha,\alpha}}$ and $D_2^{t_{\alpha,\beta}}$ have been extensively studied in \cite{IoppoloKoshlukovLaMattina2021, IoppoloKoshlukovLaMattina2021proc}. In this section we collect the results concerning their trace $T$-ideals. Such results will be employed in the next section, in order to obtain the generators of the trace $T$-ideals of the algebra $D_3$ endowed with all possible degenerate traces. 

Let us start with the case of $D_2^{t_{\alpha,0}}$. Recall that, if $\alpha = 0$, then $D_2^{t_{0,0}}$ is the algebra $D_2$ with zero trace. So $\Id^{tr}(D_2^{t_{0,0}})$ is generated by $[x_1, x_2], \tr(x)$ and $c_n^{tr}(D_2^{t_{0,0}}) = c_n(D_2^{t_{0,0}}) = 1$.  

When $\alpha \neq 0$ we will need the following result.

\begin{Theorem}  \cite[Theorem 8]{IoppoloKoshlukovLaMattina2021}  \label{identities and codimensions of D2 t alpha 0}
The trace $T$-ideal $\Id^{tr} \left (D_2^{t_{\alpha, 0}} \right)$ is generated, as a trace $T$-ideal, by the polynomials:
\begin{itemize}
\item[•] $ f_1 = [x_1,x_2]$,
\item[•] $f_2 = \Tr(x_1) \Tr(x_2) - \alpha \Tr(x_1 x_2)$.
\end{itemize}
Moreover 
\[
c_n^{tr} \left( D_2^{t_{\alpha, 0}} \right) =  2^{n}.
\]
\end{Theorem}
Similarly we need the following theorems from \cite{IoppoloKoshlukovLaMattina2021}.
\begin{Theorem} \cite[Theorem 9]{IoppoloKoshlukovLaMattina2021}  \label{identities of D2 talpha alpha}
The trace $T$-ideal $\Id^{tr}\left(D_2^{t_{\alpha, \alpha}}\right)$ is generated, as a trace $T$-ideal, by the polynomials:
\begin{itemize}
\item[•] $ f_1 = [x_1,x_2]$,
\item[•] $ f_3 = \Tr(x_1)\Tr(x_2) + \alpha^2 x_1x_2 + \alpha^2 x_2 x_1 - \alpha \Tr(x_1)x_2 - \alpha \Tr(x_2)x_1 - \alpha \Tr(x_1 x_2) $.
\end{itemize}
Moreover 
\[
c_n^{tr} \left( D_2^{t_{\alpha, \alpha}} \right)= 2^n.
\]
\end{Theorem} 

\begin{Theorem} \cite[Theorem 11]{IoppoloKoshlukovLaMattina2021}  \label{identities and codimensions of D2 talpha beta}
The trace $T$-ideal $\Id^{tr} \left( D_2^{t_{\alpha, \beta}} \right) $ is generated, as a trace $T$-ideal, by the polynomials:
\begin{align*}
f_1 &= [x_1,x_2],\\
f_4 &= \Tr(x_1) \Tr(x_2) x_3 - \Tr(x_1) \Tr(x_2 x_3) - x_1 \Tr(x_2) \Tr(x_3) +  \Tr(x_1 x_2) \Tr(x_3) \\ &- (\alpha + \beta)  \Tr(x_1 x_2) x_3 + (\alpha + \beta) x_1 \Tr(x_2 x_3) ,\\
f_5 &= \Tr(x_1) \Tr(x_2) \Tr(x_3) - (\alpha + \beta)  \Tr(x_1) \Tr(x_2 x_3)  - (\alpha + \beta) x_1 \Tr(x_2) \Tr(x_3) \\ 
&+ \alpha \beta  \Tr(x_1) x_2 x_3 + \alpha \beta x_1 x_3 \Tr(x_2) + \alpha \beta x_1 x_2 \Tr(x_3) \\
&+ \alpha \beta \Tr(x_1 x_2 x_3) - \alpha \beta x_3 \Tr(x_1 x_2)  - \alpha \beta x_2 \Tr(x_1 x_3) \\
&+ (\alpha^2 + \alpha \beta + \beta^2) x_1 \Tr(x_2 x_3)  - (\alpha \beta^2 + \alpha^2 \beta) x_1 x_2 x_3.
\end{align*}
Moreover 
\[
c_n^{tr} \left( D_2^{t_{\alpha, \beta}} \right) =  2^{n+1}-n-1.
\]
\end{Theorem}

\section{The algebra $D_3$ endowed with a degenerate trace}

In this section we deal with the algebra $D_3$ of $3 \times 3$ diagonal matrices over the field $F$ endowed with all possible degenerate traces. By taking into account the results of Section 3, it is easy to see that on $D_3$, up to isomorphism, it is possible to define the following kinds of degenerate trace functions:
\begin{itemize}
\item[1.] 
$t_{\alpha, 0, 0}$, for every $\alpha \in F $, 
\item[2.] 
$t_{\alpha, \alpha, 0}$, for every non-zero $\alpha \in F $, 
\item[3.] 
$t_{\alpha, \beta, 0}$, for every distinct non-zero $\alpha$, $\beta \in F $.
\end{itemize}

We shall use the results of Section 4 (in particular Theorem \ref{identities of Dn+1 from those of Dn}) in order to obtain the generators of the trace $T$-ideals of such algebras starting from those of the algebra $D_2$ with the corresponding trace (presented in the previous section).

Let us begin with the algebra $D_3^{t_{\alpha,0,0}}$. Notice that if $\alpha = 0$, then $D_3^{t_{0,0,0}}$ is the algebra $D_3$ with zero trace. So in this case $\Id^{tr}(D_3^{t_{0,0,0}})$ is generated by the commutator $[x_1, x_2]$ and $\tr(x)$. For the codimensions we have $c_n^{tr}(D_3^{t_{0,0,0}}) = c_n(D_3^{t_{0,0,0}}) = 1$. 

Now suppose that $\alpha \neq 0$. In \cite[Theorem 6]{IoppoloKoshlukovLaMattina2021proc} it was proved that $\Id^{tr}(D_3^{t_{\alpha,0,0}})=\Id^{tr}(D_2^{t_{\alpha,0}}).$

Actually it was given the following  general result. 

\begin{Theorem} \label{identities and codimensions of Dn t alpha 0 ... 0}
Let $\alpha \in F \setminus \{ 0 \}$. Then $\Id^{tr}	\left ( D_n^{t_{\alpha, 0, \ldots, 0}} \right ) $ is generated, as a trace $T$-ideal, by the polynomials:
\begin{itemize}
\item[•] $f_1 = [x_1,x_2]$,
\item[•] $f_2 = \Tr(x_1) \Tr(x_2) - \alpha \Tr(x_1 x_2)$.
\end{itemize}
Moreover 
\[
c_n^{tr}(D_n^{t_{\alpha, 0, \ldots, 0}}) =  2^{n}.
\]
\end{Theorem}

Concerning the algebra $D_3^{t_{\alpha, \alpha, 0}}$, the generators of the trace $T$-ideal of identities have been already found in \cite[Theorem 7]{IoppoloKoshlukovLaMattina2021proc} by using different methods. Here we prove once again such a result by making use of Theorem \ref{identities of Dn+1 from those of Dn}. Moreover we compute the trace codimension sequence of the algebra $D_3^{t_{\alpha, \alpha, 0}}$. 

First recall that the Stirling numbers of the second type $S(n, k)$ count the ways to partition a set of $n$ objects into $k$ non-empty subsets (see, for instance, \cite{GrahamKnuthPatashnik1994} for more details). In the following result we collect some well-known properties of the Stirling numbers of the second type.

\begin{Lemma} \label{stirling numbers properties}
For the Stirling numbers of the second type we have that:
\begin{itemize}
\item[1.] 
$\displaystyle S(n, k) = \dfrac{1}{k!} \sum_{i= 0}^k (-1)^{k-i} \binom{k}{i} i^{n}$,
\item[2.] 
$\displaystyle S(n+1, k+1) = \sum_{t=k}^n \binom{n}{t} S(t,k)$.
\end{itemize}
\end{Lemma}

\begin{Theorem} \label{identities and codimensions of D3 t alpha alpha 0}
Let $\alpha \in F \setminus \{ 0 \}$. The trace $T$-ideal $\Id^{tr} \left( D_3^{t_{\alpha, \alpha, 0}} \right)$ is generated, as a trace $T$-ideal, by the polynomials:
\begin{itemize}
\item[•] 
$f_1 = [x_1,x_2]$,
\item[•] 
$g_3 = \Tr(x_1) \Tr(x_2) \Tr(x_3) + 2 \alpha^2 \Tr(x_1 x_2 x_3) - \alpha \Tr(x_1) \Tr(x_2 x_3) - \alpha \Tr(x_2) \Tr(x_1 x_3) - \alpha \Tr(x_1 x_2) \Tr(x_3)$.
\end{itemize}
Moreover 
\[
c_n^{tr} \left( D_3^{t_{\alpha, \alpha, 0}} \right) =  \sum_{k=0}^n \binom{n}{k} + 
\sum_{t=2}^n \binom{n}{t} S(t,2) = \dfrac{3^n+1}{2}.
\]
\end{Theorem}
\begin{proof}
Let us consider the polynomial $f_3(x_1, x_2)$ given in Theorem \ref{identities of D2 talpha alpha}. It is easy to see that 
\[
g_3 = \Tr (f_3 (x_1, x_2) x_3).
\] 
Now the first part follows directly from Theorem \ref{identities of Dn+1 from those of Dn}.

We are left to compute the trace codimension sequence. By taking into account the identities $f_1$ and $g_3$ it is easy to prove that the following polynomials form a basis of $MT_n \left( \mbox{mod } MT_n \cap \Id^{tr} \left( D_3^{t_{\alpha, \alpha, 0}} \right) \right)$: 
\begin{equation} \label{monomials with one trace}
\Tr (x_{i_1} \cdots x_{i_k}) x_{j_1} \cdots x_{j_{n-k}}, \ \ \ \ \ k = 0, \ldots, n, \ \ \ i_1 < \cdots < i_k, \ j_1 < \cdots < j_{n-k},
\end{equation}
\begin{equation} \label{monomials with 2 traces}
\Tr (x_{p_1} \cdots x_{p_a}) \Tr (x_{q_1} \cdots x_{q_b}) x_{r_1} \cdots x_{r_{n-a-b}}, \ \ \ \ \ a \geq b \geq 1, \ \ \ p_1 < \cdots < p_a, \ q_1 < \cdots < q_b, \ r_1 < \cdots < r_{n-a-b}.
\end{equation}
It is easy to see that there are exactly $\sum_{k = 0}^n \binom{n}{k} = 2^n$  distinct monomials in \eqref{monomials with one trace}. Now let us focus our attention to the monomials in \eqref{monomials with 2 traces}. Let us denote by $t = a+b \geq 2$ the quantities of variables appearing inside the traces. There are exactly $\binom{n}{t}$ distinct ways to choose the elements appearing inside the traces. For fixed $t$ elements, we have exactly $S(t,2)$ ways to put them inside the two traces. Hence the quantity of monomials in \eqref{monomials with 2 traces} is $\sum_{t=2}^n \binom{n}{t} S(t,2)$. In conclusion
\[
c_n^{tr} \left( D_3^{t_{\alpha, \alpha, 0}} \right) =  \sum_{k=0}^n \binom{n}{k} + 
\sum_{t=2}^n \binom{n}{t} S(t,2).
\]
Now, by Lemma \ref{stirling numbers properties} we obtain that 
\[
c_n^{tr} \left( D_3^{t_{\alpha, \alpha, 0}} \right) = 2^n + S(n+1, 3)  = 2^n + \dfrac{1}{2} \left( 3^n - 2^{n+1} +1 \right) = \dfrac{3^n+1}{2}.
\]
Thus the proof of the theorem is complete.
\end{proof}

By using the same technique, one can prove the following more general result. 

\begin{Remark} \label{number of trace monomials and Stirling numbers}
The number of trace monomials in $n$ commuting variables $x_1$, \dots, $x_n$ with exactly $k$ traces equals 
\[
S(n+1, k+1) = \sum_{t=k}^n \binom{n}{t} S(t,k).
\]
\end{Remark}

We conclude this section by considering the case of $D_3^{t_{\alpha, \beta, 0}}$, $\alpha$, $\beta \neq 0$, $\alpha \neq \beta$. In \cite{IoppoloKoshlukovLaMattina2021proc}, the authors presented just some trace identities satisfied by this algebra but no basis of the trace identities was exhibited. Here we are able to address that problem; we prove the following theorem. 

\begin{Theorem} \label{identities of D3 t alpha beta 0}
The trace $T$-ideal $\Id^{tr} \left( D_3^{t_{\alpha, \beta, 0}} \right)$ is generated, as a trace $T$-ideal, by the polynomials:
\begin{align*}
f_1 &= [x_1,x_2],\\
g_4 &= \Tr(x_1) \Tr(x_2) \Tr(x_3 x_4) - \Tr(x_1)\Tr(x_4)\Tr(x_2 x_3) - \Tr(x_2)\Tr(x_3)\Tr(x_1 x_4) + \Tr(x_3) \Tr(x_4) \Tr(x_1 x_2) \\ 
&- (\alpha + \beta)\Tr(x_1 x_2)\Tr(x_3 x_4) + (\alpha + \beta) \Tr(x_1 x_4) \Tr(x_2 x_3),\\
g_5 &=  \Tr(x_1) \Tr(x_2) \Tr(x_3) \Tr(x_4) -(\alpha+\beta)\Tr(x_1)\Tr(x_4)\Tr(x_2 x_3) - (\alpha+\beta)\Tr(x_2)\Tr(x_3)\Tr(x_1 x_4) \\
&+ \alpha \beta \Tr(x_1) \Tr(x_2 x_3 x_4) 
+ \alpha \beta \Tr(x_2) \Tr(x_1 x_3 x_4) 
+ \alpha \beta \Tr(x_3) \Tr(x_1 x_2 x_4)  \\
&+ \alpha \beta \Tr(x_4) \Tr(x_1 x_2 x_3)
- \alpha \beta \Tr(x_1 x_2)\Tr(x_3 x_4) 
- \alpha \beta \Tr(x_1 x_3)\Tr(x_2 x_4) \\
&+( \beta^2 + \alpha \beta +  \alpha^2) \Tr(x_1 x_4) \Tr(x_2 x_3) 
-(\alpha \beta^2+\alpha^2 \beta)\Tr(x_1 x_2 x_3 x_4). 
\end{align*}
Moreover the codimensions are given by 
\[
c_n^{tr} \left( D_3^{t_{\alpha, \beta, 0}} \right) =  2^n + 
S(n+1,3) + n S(n,3) = \dfrac{1}{2} \left( 3^n + 3^{n-1}n - 2^n n + n +1 \right).
\]
\end{Theorem}
\begin{proof}
Consider the polynomials $f_4$ and $f_5$ in Theorem 	\ref{identities and codimensions of D2 talpha beta}. The first part follows by Theorem \ref{identities of Dn+1 from those of Dn} since one can easily check that 
\[
g_4 = \Tr \left( f_4(x_1, x_2, x_3) x_4 \right) \qquad \mbox{and} \qquad g_5 = \Tr \left( f_5(x_1, x_2, x_3) x_4 \right).
\]
We are left to compute the trace codimension sequence. By using the same approach as in \cite[Theorem $11$]{IoppoloKoshlukovLaMattina2021}, it is not difficult to see that the following polynomials form a basis of $MT_n \left( \mbox{mod } MT_n \cap \Id^{tr} \left( D_3^{t_{\alpha, \beta, 0}} \right) \right)$: 
\begin{equation} \label{monomials with one trace t a,b,0}
\Tr (x_{i_1} \cdots x_{i_k}) x_{j_1} \cdots x_{j_{n-k}}, \ \ \ \ \ k = 0, \ldots, n, \ \ \ i_1 < \cdots < i_k, \ j_1 < \cdots < j_{n-k},
\end{equation}
\begin{equation} \label{monomials with 2 traces t a,b,0}
\Tr (x_{p_1} \cdots x_{p_a}) \Tr (x_{q_1} \cdots x_{q_b}) x_{r_1} \cdots x_{r_{n-a-b}}, \ \ \ \ \ a \geq b \geq 1, \ \ \ p_1 < \cdots < p_a, \ q_1 < \cdots < q_b, \ r_1 < \cdots < r_{n-a-b},
\end{equation}
\begin{equation} \label{monomials with 3 traces t a,b,0}
\Tr( x_s ) \Tr (x_{u_1} \cdots x_{u_c}) \Tr (x_{v_1} \cdots x_{v_d}) x_{t_1} \cdots x_{t_{e}}, \ \ c \geq d \geq 1, \ \ u_1 < \cdots < u_c, \ \ v_1 < \cdots < v_d, \ \ t_1 < \cdots < t_{e}.
\vspace{0.2 cm}
\end{equation}
As we have seen in Theorem \ref{identities and codimensions of D3 t alpha alpha 0}, the number of monomials in \eqref{monomials with one trace t a,b,0} and \eqref{monomials with 2 traces t a,b,0} is $2^n$ and $S(n+1,3)$, respectively. In order to complete the proof we need to show that in \eqref{monomials with 3 traces t a,b,0} there are exactly $nS(n,3)$ elements. 
Notice that we can choose the variable $x_s$ in $n$ distinct ways. Once $x_s$ is selected, we are left with monomials in $n-1$ variables and exactly two traces. Hence their number is $S(n,3)$ and the proof is complete.
\end{proof}

\section{The algebras $C_k^{t_{\alpha,\beta}}$}

In \cite{IoppoloKoshlukovLaMattina2021} we introduced the $F$-algebra
\[
C_2 = 	\left \{ \begin{pmatrix} 
a & b \\
0 & a
\end{pmatrix} \mid a,b \in F \right \}
\]
endowed with the trace $t_{\alpha, \beta}$ defined by
\[
t_{\alpha, \beta} \left ( \begin{pmatrix} 
a & b \\
0 & a
\end{pmatrix} \right ) = \alpha a + \beta b.
\]
We recall that in \cite[Section 2]{Procesi2008} such traces were called \textsl{strange traces}. It is immediate to see that the algebra $C_2$ is isomorphic to the truncated polynomial algebra $F[x]/(x^2)$. 

In \cite{IoppoloKoshlukovLaMattina2021}, we proved that, for any $\beta \neq 0$, $C_2^{t_{\alpha,\beta}}$ is isomorphic as an algebra with trace to $C_2^{t_{\alpha,1}}$. So in order to determine $\Id^{tr} \left( C_2^{t_{\alpha, \beta}} \right)$ it is sufficient to study $C_2^{t_{\alpha, 1}}$. In the following results we shall describe the generators of its trace $T$-ideal. We consider first the case in which $\alpha = 0$. Then we reduce the general case to this particular one.

\begin{Theorem} \label{identities of C2 t 0,1}
The trace $T$-ideal $\Id^{tr} \left( C_2^{t_{0, 1}} \right)$ is generated, as a trace $T$-ideal, by the polynomials:
\begin{itemize}
\item[•] 
$f_1 = [x_1,x_2]$,
\item[•] 
$h_2 = \Tr(x_1) \Tr(x_2)x_3 - \Tr(x_2) \Tr(x_3)x_1 + \Tr(x_1x_2) \Tr(x_3) - \Tr(x_2x_3) \Tr(x_1)$,
\item[•] 
$h_3 = -\Tr(x_1x_2)x_3 - \Tr(x_1x_3)x_2 - \Tr(x_2x_3)x_1 + \Tr(x_1x_2x_3) +\Tr(x_1)x_2x_3 + \Tr(x_2)x_1x_3 + \Tr(x_3)x_1x_2$.
\end{itemize}
Moreover 
\[
c_n^{tr} \left( C_2^{t_{0, 1}} \right) = 2^{n+1} - n -1.
\]
\end{Theorem}
\begin{proof}
A straightforward computation shows that the polynomials $f_1$, $h_2$, and $h_3$ are trace identities for $C_2^{t_{0, 1}}$, hence the inclusion $T = \langle f_1, h_2, h_3 \rangle_{T^{tr}} \subseteq \Id^{tr}(C_2^{t_{0, 1}})$ holds. 

In order to obtain the opposite inclusion, first we shall prove that the polynomials
\begin{equation} \label{non identities of C2 t 1,0}
\Tr(x_{i_1} x_{i_2}) \Tr(x_{i_3}) \cdots \Tr(x_{i_k})  x_{j_1} \cdots x_{j_{m}}, \ \ \ \ \ k+m = n, \ \ \ i_1 < \cdots < i_k, \ \ j_1 < \cdots < j_{m}
\end{equation}
span $MT_n$, modulo $MT_n \cap T$, for every $n \geq 1$. Clearly any one of the three parts (trace of degree $2$, traces with just one variable or variables outside the traces) can be missing. 

By the identity $f_1$ we can order all the variables, inside and outside a trace. By using the identity $h_3$ one gets rid of all traces of monomials of degree $\ge 3$. Hence one is left with products of traces of degree 1 and 2 (and some variables outside traces). Now we use the identity $h_2$: suppose $x_1=uv$ where $u$ and $v$ are variables, and consider the last summand in $h_2$ (which is a product of two traces of degree $2$). The last but one summand becomes $\Tr(uvx_2) \Tr(x_3)$. The first two summands are inoffensive since they contain exactly 2 traces and at most one of those contains 2 variables. Now identity $h_3$ applied to $\Tr(uvx_2)$ gives us only one trace of degree 2. In conclusion we can assume that we have only one trace of degree $2$ (or none), several traces of degree $1$ (or none), and several variables outside traces (or none). By applying the identity $h_2$ (and $f_1$ when necessary), we can take the variable $x_1$ outside the trace of degree 2 and put it into a trace of degree 1 (look at the two last summands). Therefore we have proved that the polynomials in \eqref{non identities of C2 t 1,0} span $MT_n$, modulo $MT_n \cap T$, as desired.

Our next goal is to show that the elements in \eqref{non identities of C2 t 1,0} are linearly independent modulo $ \Id^{tr}(C_2^{t_{0, 1}})$. To this end let 
\[
x_i = \begin{pmatrix}
a_i & b_i \\
0 & a_i
\end{pmatrix}
\]
be generic matrices in $C_2^{t_{0,1}}$. We suppose that $a_i$ and $b_i$ are commuting independent variables. In order to reach our goal, it suffices to show that no non-trivial linear combination of the above elements, evaluated on the generic matrices $x_i$, $1\le i\le n$, vanishes. 

Fix a trace monomial of length $n$. We know that it is of the types
\begin{align*}
u(p_1,\ldots,p_k) &= \Tr(x_{p_1})\cdots \Tr(x_{p_k}) x_{q_1}\cdots x_{q_m},\\
v(p,p_1,\ldots, p_k) &= \Tr(x_p x_{p_1}) \Tr(x_{p_2})\cdots \Tr(x_{p_k}) x_{r_1}\cdots x_{r_s}
\end{align*}
where we have $p_1<\cdots<p_k$, $q_1<\cdots< q_m$, $\{p_1,\ldots,p_k\}\cup \{q_1,\ldots,q_m\} =\{1,\ldots,n\}$. Similarly in the second element we require $p<p_1<\cdots<p_k$, $r_1<\cdots<r_s$, and $\{p,p_1,\ldots,p_k\}\cup \{r_1,\ldots,r_s\}=\{1,\ldots,n\}$, and the unions are disjoint. Clearly all these imply $k+m=n$ for the first and $k+1+s=n$ for the second element.

Suppose we have a linear combination of the above trace elements which is 0 when evaluated on the generic matrices, then it is a trace identity for our algebra. Suppose the combination is not trivial. Let $u(p_1,\ldots,p_k)$ enter in it with coefficient $\alpha(p_1,\ldots,p_k)$, and let $v(p,p_1,\ldots, p_k)$ enter with coefficient $\beta(p,p_1,\ldots, p_k)$.

Evaluating such a linear combination of the above monomials on the generic matrices $x_i$ we obtain a matrix whose entries are polynomials in the variables $a_i$ and $b_i$; clearly the entries at positions $(1,1)$ and $(2,2)$ are equal. Consider a monomial in the $a_i$ and $b_i$, we shall count the entries $b$ in it. We shall consider what appears at the diagonal of the resulting matrix (say at position (1,1)), and afterwards at position (1,2). Clearly as the base field is of characteristic 0 (infinite), by homogeneity we can consider a fixed set of letters $b_i$. 

There are two trivial cases for the monomials at position $(1,1)$. The first is when one has no letter $b$ at all. This monomial comes only from $x_1\cdots x_n$ hence the latter product cannot participate in our combination. In the second case we have the monomial $b_1\cdots b_n$; it comes only from $\Tr(x_1)\cdots \Tr(x_n)$, hence it does not enter our linear combination either.

Let us first consider $u(i_1,\ldots,i_k)$. It will give at position (1,1) the monomial $b_{i_1}\cdots b_{i_k}$. Here we shall omit altogether the product $a_{d_1}\cdots a_{d_m}$ where $k+m=n$, $\{i_1,\ldots,i_m\}\cup \{d_1,\ldots,d_m\}=\{1,\ldots,n\}$. We list all trace monomials that will produce $b_{i_1}\cdots b_{i_k}$ at position (1,1).

\begin{align*}
u(i_1,\ldots,i_k) &\mapsto b_{i_1}\cdots b_{i_k}\\
v(i,i_1,\ldots, i_k)&\mapsto b_{i_1}\cdots b_{i_k}+ b_i b_{i_2}\cdots b_{i_k}, \quad i<i_1\\
v(i_1,i,\ldots, i_k)&\mapsto b_{i_1}\cdots b_{i_k}+b_ib_{i_2}\cdots b_{i_k}, \quad i_1<i<i_2.
\end{align*}
In order to kill the monomial $b_{i_1}\cdots b_{i_k}$ at position (1,1) we must have
\begin{equation}
\label{11}
\alpha(i_1,\ldots,i_k) +\sum_{i=1}^{i_1-1} \beta(i,i_1,\ldots,i_k) + \sum_{i=i_1+1}^{i_2-1} \beta(i_1,i,i_2,\ldots,i_k) = 0.
\end{equation}

Now we analyse the corresponding entries at position (1,2). The two trivial cases discarded (see above), we will have exactly one additional letter $b$. Let $j\notin\{i_1,\ldots,i_k\}$, and consider the monomial $b_{i_1}\cdots b_{i_k}b_j$ at position (1,2). We require $i_1<\cdots<i_k$ only. Fix such an index $j$.

The monomial $b_{i_1}\cdots b_{i_k}b_j$ at position (1,2) comes from the monomial $u(i_1,\ldots, i_k)$, and from the following elements: 

\begin{enumerate}
\item[(a)]
If $j<i_1$, from $u(j, i_2, \ldots, i_k)$;
\item[(b)]
If $j>i_k$, from $u(i_1, \ldots,i_{k-1}, j)$;
\item[(c)]
If $i_{s-1}<j<i_s$ for some $2\le s\le k$, from 
 $u(i_1,\ldots, i_{s-2}, j, i_s,\ldots i_k)$, and $u(i_1,\ldots, i_{s-1}, j, i_{s+1}, \ldots,i_k)$. (Here we plug the index $j$ in place of $i_{s-1}$ and $i_s$, respectively.)
\end{enumerate}

The monomial $b_{i_1}\cdots b_{i_k}b_j$ can also come from the elements
\begin{itemize}
\item[1.]
$v(i,i_1,\ldots,i_k)$, $i<i_1$ and $i\ne j$; 
\item[2.]
$v(i_1,i,i_2,\ldots,i_k)$, $i_1<i<i_2$ and $i\ne j$;
\item[3.]
$v(i,j,i_2,\ldots,i_k)$, $i<j<i_2$ and $i\ne i_1$;
\item[4.]
$v(j,i,i_2,\ldots,i_k)$, $j<i<i_2$ and $i\ne i_1$;
\item[5.]
$v(i,i_1,\ldots, i_{s-2}, j, i_s,\ldots,i_k)$, $i<i_1$, if $i_{s-1}<j<i_s$;
\item[6.]
$v(i_1,i,\ldots, i_{s-2}, j, i_s,\ldots,i_k)$, $i_1 < i < i_2$, if $i_{s-1}<j<i_s$;
\item[7.]
$v(i_1,i,\ldots, i_{s-1}, j, i_{s+1},\ldots,i_k)$, $i_1<i<i_2$ if $i_{s-1}<j<i_s$.
\item[8.]
$v(i,i_1,\ldots, i_{s-1}, j, i_{s+1},\ldots,i_k)$, $i<i_1$ if $i_{s-1}<j<i_s$.
\end{itemize}
Clearly the elements from (1) and (2) enter whenever there exist indices $i$ with the required inequalities. Depending on the position of $j$ among the $i_1$, \dots, $i_k$ we will have (3) or (4) or (5) or (6) or (7) or (8), for some $s$.

Now let $j<i_1$, then we get the equality 
\begin{align*}
\alpha(i_1,\ldots, i_k) +\alpha(j,i_2,\ldots,i_k) + \sum_{i=1}^{i_1-1} \beta(i,i_1,\ldots,i_k) &- \beta(j, i_1,\ldots,i_k) \\
+\sum_{i=i_1+1}^{i_2-1} \beta(i_1,i,i_2,\ldots,i_k) + \sum_{i=1}^{j-1} \beta(i,j, i_2,\ldots,i_k) &+ \sum_{i=j+1}^{i_2-1} \beta(j,i,i_2,\ldots,i_k) - \beta(j, i_1,\ldots,i_k) =0.
\end{align*}
(In the first line we subtract $\beta(j,i_1,\ldots,i_k)$ as it cannot appear: it does not produce the desired element at position (1,2) and the same happens in the second line). 
Subtract from the latter equation the one from (\ref{11}), and then the same from (\ref{11}) but with the indices $(j,i_2,\ldots,i_k)$ instead of $(i_1,\ldots,i_k)$. 
Then every term cancels except for $-2\beta(j,i_1,\ldots,i_k)$. Hence the latter coefficient must be 0.

In a similar way we deal with the case when $i_1<j<i_2$. We obtain  
\begin{align*}
\alpha(i_1, \ldots, i_k) + \alpha(j,i_2,\ldots,i_k) + \alpha(i_1, j, i_3, \ldots, i_k) &+\sum_{i=1}^{i_1-1} \beta(i, i_1, \ldots, i_k)+\sum_{i=i_1+1}^{i_2-1} \beta(i_1, i, i_2, \ldots, i_k)\\
-\beta(i_1, j, i_2, \ldots, i_k) + \sum_{i=1}^{j-1} \beta(i, j, i_2, \ldots, i_k) &- \beta(i_1, j, i_2, \ldots, i_k) + \sum_{i=j+1}^{i_2-1} \beta(j, i, i_2, \ldots, i_k) \\
+ \sum_{i=1}^{i_1-1} \beta(i, i_1, j, i_3, \ldots, i_k) &+\sum_{i=i_1+1}^{j-1} \beta(i_1, i, j, i_3, \ldots, i_k)=0.
\end{align*}
Proceeding as above we subtract from that equality the ones obtained from position $(1,1)$ for the monomials $b_{i_1}\cdots b_{i_k}$, $b_j b_{i_2} \cdots b_{i_k}$, and $b_{i_1}b_jb_{i_3}\cdots b_{i_k}$, we will be left with $-2\beta(i_1, j, i_2, \ldots, i_k) =0$. 

Now we recall what was proved above, this will be important in what follows. 
\begin{itemize}
\item
We start with a monomial $b_{p_1}\cdots b_{p_k}$ that appears at position $(1,1)$.
\item
We show that the coefficients $\beta(p,p_1,\ldots, p_k)=0$ for $p<p_1$.
\item
We show that the coefficients $\beta(p_1,p, p_2,\ldots, p_k)=0$ for $p$ such that $p_1<p<p_2$. 
\end{itemize}

As a consequence, in Equation (\ref{11}) all the coefficients $\beta$ are equal to $0$ and so $\alpha(i_1, \ldots,i_k)=0$.

Suppose now that $i_1<i_2<\cdots<i_k$, $j\notin\{i_1, \ldots,i_k \}$ and $j>i_2$. Then in Equation (\ref{11}) we have coefficients $\beta$ where we plug in some $i$ which satisfies either $i<i_1$ or $i_1<i<i_2$. But this implies all these coefficients $\beta$ are equal to 0. Therefore we get also in this case that $\alpha(i_1, \ldots,i_k)=0$. Then one replaces the set $\{i_1,\ldots,i_k\}$ with another appropriate set of $k$ positive integers and continues until all coefficients in the linear combination turn out to equal 0. 

We observe that the following cases were formally left outside the scope of the above argument. 
\begin{itemize}
\item
$i_1=1$, then one cannot plug in any $i<i_1$ in the coefficient $\beta(i,i_1,\ldots,i_k)$ in Eq.~(\ref{11}). If this is the case but $i_2>2$ then the argument is the same as above.
\item
$i_1=1$ and $i_2=2$. In this case we get immediately $\alpha(1,2,i_3,\ldots, i_k)=0$, and we are done as the corresponding trace monomial does not participate in the linear combination.
\end{itemize}

In conclusion the elements in \eqref{non identities of C2 t 1,0} are linearly independent modulo $ \Id^{tr}(C_2^{t_{0, 1}})$ and so $T = \Id^{tr} \left( C_2^{t_{0,1}} \right)$.

We are left to compute the trace codimension sequence. Since the monomials in \eqref{non identities of C2 t 1,0} which do not contain traces of length $2$ are exactly $\sum_{k=0}^n \binom{n}{k}$ and the remaining ones are $\sum_{k=2}^n \binom{n}{k}$, we get that
\[
c_n^{tr} \left( C_2^{t_{0, 1}} \right) =  \sum_{k=0}^n \binom{n}{k} + 
\sum_{k=2}^n \binom{n}{k}  = 2^{n+1} - n -1.
\]
\end{proof}

Now we can consider the case of $C_2^{t_{\alpha,1}}$, where $\alpha \in F$ can be non-zero. 

\begin{Theorem} \label{identities of C2 t a,1}
Let $\alpha \in F$. The trace $T$-ideal $\Id^{tr} \left( C_2^{t_{\alpha, 1}} \right)$ is generated, as a trace $T$-ideal, by the polynomials:
\begin{itemize}
\item[•] 
$f_1 = [x_1,x_2]$,
\item[•] 
$h_4 = \Tr(x_1) \Tr(x_2)x_3 - \Tr(x_2) \Tr(x_3)x_1 + \Tr(x_1x_2) \Tr(x_3) - \Tr(x_2x_3) \Tr(x_1) - \alpha \Tr(x_1 x_2) x_3 + \alpha \Tr(x_2 x_3) x_1$,
\item[•] 
$h_5 = -\Tr(x_1x_2)x_3 - \Tr(x_1x_3)x_2 - \Tr(x_2x_3)x_1 + \Tr(x_1x_2x_3) +\Tr(x_1)x_2x_3 + \Tr(x_2)x_1x_3 + \Tr(x_3)x_1x_2 - \alpha x_1 x_2 x_3$.
\end{itemize}
Moreover 
\[
c_n^{tr} \left( C_2^{t_{\alpha, 1}} \right) = 2^{n+1} - n -1.
\]
\end{Theorem}
\begin{proof}
The case $\alpha = 0$ was proved in the previous theorem. So assume $\alpha \neq 0$, we shall reduce this case to that of $\alpha=0$. The same arguments of the first part of the proof of Theorem \ref{identities of C2 t 0,1} show that the polynomials in \eqref{non identities of C2 t 1,0} span $MT_n$, modulo $MT_n \cap T$, where $T = \langle f_1, h_4, h_5 \rangle_{T^{tr}}$. 

In order to complete the proof it is sufficient to prove that these elements are linearly independent. We use also in this case generic matrices. Recall that $\Tr(x_i)= \alpha a_i+b_i$. Then each of the trace monomials from the proof of the previous theorem, that  yields the element $b_{i_1}\cdots b_{i_k}$ will produce some other elements with fewer $b$'s in it, their indices forming some proper subset of $\{i_1,\ldots,i_k\}$. Hence we can consider the monomials at position $(1,1)$ with the largest possible degree, say $k$, in the letters $b$. Then we fix a monomial that appears in our combination, say $b_{i_1}\cdots b_{i_k}$, and repeat the above argument.
\end{proof}

Our next goal is to study the algebras $C_2^{t_{\alpha,0}}$, i.e., the algebra $C_2$ endowed with the degenerate trace $t_{\alpha, 0}$. Actually, we shall introduce a generalization of such an algebra. Let us consider, for any $k\geq 1,$
\[
C_k = \left \{
\begin{pmatrix}
a_1 & a_2 & \cdots & a_k \\
0 & a_1 & \ddots & \vdots \\
\vdots  & \vdots  & \ddots & a_2  \\
0 & 0 & \cdots & a_1 
\end{pmatrix} \mid a_1, \ldots, a_k \in F \right \}.
\]
Notice that for $k>1$ the Jacobson radical of $C_k$  $J = J(C_k) = F (e_{12} + e_{23} + \cdots + e_{k-1k}) + \cdots + F (e_{1k})
$ consists of strictly upper triangular matrices, and hence is such that $J^k = 0$. One may interpret the algebra $C_k$ as the quotient $F[x]/(x^k)$, that is the truncated polynomial algebra, in an obvious manner. 

We define the trace $t_{\alpha, 0, \ldots, 0}$ on the algebra $C_k$ as follows: 
\[
t_{\alpha, 0, \ldots, 0}(a(e_{11} + \cdots + e_{kk}) + b (e_{12} + e_{23} + \cdots + e_{k-1k}) + \cdots + c (e_{1k})) = \alpha a + 0b + \cdots + 0c = \alpha a.
\]
To simplify the notation, we shall use the symbol $C_k^{t_{\alpha,0}}$ to indicate such a trace algebra.

\begin{Theorem} \label{identities of Ck t alfa 0}
Let $\alpha \in F \setminus \{0\}$, $k \geq 2 $. Then $\Id^{tr} \left( C_k^{t_{\alpha, 0}} \right)$ is generated, as a trace $T$-ideal, by the polynomials:
\begin{itemize}
\item[•] $f_1 = [x_1, x_2]$.
\item[•] $g_6 = \Tr(x_1) \Tr(x_2) - \alpha \Tr(x_1 x_2)$.
\item[•] $g_7 = \left( \Tr(x_1) - \alpha x_1 \right) \cdots \left( \Tr(x_k) - \alpha x_k \right)$.
\end{itemize}
Moreover, $ \displaystyle c_n^{tr} \left( C_k^{t_{\alpha, 0}} \right) = \sum_{i = 0}^{k-1} \binom{n}{i} \approx \dfrac{n^{k-1}}{(k-1)!} $.
\end{Theorem}
\begin{proof}
It is clear that $f_1$ and $g_6$ are trace identities of $C_k^{t_{\alpha, 0}}$. Now for each $a+j \in C_k^{t_{\alpha,0}}$, $a \in F(e_{11} + \cdots + e_{kk})$ and $j \in J$, we have that
\[
t_{\alpha, 0, \ldots, 0}(a+j) - \alpha (a+j) = - \alpha j \in J.
\]
Since $J^k = 0$, we get that $g_7 \equiv 0$ on $C_{k}^{t_{\alpha,0}}$. 

So far we have proved the inclusion $T = \langle f_1, g_6, g_7 \rangle_{T^{tr}} \subseteq \Id^{tr}(C_k^{t_{\alpha, 0}})$. 

In order to obtain the opposite inclusion, first we shall prove that the polynomials
\begin{equation} \label{non identities of Ck t alpha 0}
\Tr(x_{i_1} \cdots x_{i_a}) x_{j_1} \cdots x_{j_{n-a}}, \ \ \ \ \ a = 0, \ldots, k-1, \ \ \ i_1 < \cdots < i_a, \ \ j_1 < \cdots < j_{n-a},
\end{equation}
span $MT_n$, modulo $MT_n \cap T$, for every $n \geq 1$.

By the identity $f_1$ we can order all the variables, inside and outside a trace. Moreover, because of $g_6$, we can kill all products of two traces (and more than two traces). So we may consider only monomials with either no traces (there is only one such monomial namely $x_1 \cdots x_n$), or with just one trace. Now by taking into account the identity $g_7$ (modified in accordance with the identity $g_6$) it is not difficult to see that one can have, inside the trace, at most $k-1$ variables.

Our next goal is to show that the polynomials in \eqref{non identities of Ck t alpha 0} are linearly independent modulo $ \Id^{tr}(C_k^{t_{\alpha, 0}})$.  Here we assume that $n >k$, otherwise the proof is trivial. To this end, let $g = g(x_1, \ldots, x_n, \Tr)$ be a linear combination of the above polynomials which is a trace identity:
\[
g(x_1, \ldots, x_n, \Tr) = a_0 x_1 \cdots x_n +  \sum_{i =1}^{k-1} \sum_{l_1, \ldots, l_i} a_{l_1, \ldots, l_i} \Tr(x_{l_1} \cdots x_{l_i}) x_{m_1} \cdots x_{m_{n-i}}.
\]
Here $l_1 < \cdots < l_{i}$ and $m_1 < \cdots < m_{n-i}$. Notice that we use the notation $a_{l_1, \ldots, l_i}$ to denote the coefficient of the monomial in which the variables $x_{l_1}$, \dots, $x_{l_i}$ appear inside the trace, $i = 1$, \dots, $k-1$. 

Our first goal is to show that $a_{1, \ldots, k-1} = 0$. To this end we evaluate all the variables $x_k$, \dots, $x_{n}$ into the identity matrix $I=e_{11}+ \cdots + e_{kk}$. Now we make $k$ distinct types of substitutions for the variables $x_1$, \dots, $x_{k-1}$. Denote by $b = e_{12}+ \cdots + e_{k-1,k}$, we will evaluate some of the first $k-1$ variables to $b$, and the remaining to $I$.
\begin{itemize}
\item[0)] $0$ variables go to $b$: we have $\binom{k-1}{0} = 1$ evaluations of this kind.
\item[1)] $1$ variable goes to $b$: we have $\binom{k-1}{1} = k-1$ evaluations of this kind.
\item[2)] $2$ variables go to $b$: we have $\binom{k-1}{2}$ evaluations of this kind.
\\ 
$
\vdots
$
\item[k-1)] $k-1$ variables (all of them) go to $b$: we have $\binom{k-1}{k-1} = 1$ evaluations of this kind.
\end{itemize}
In this way we obtain exactly $\sum_{i=0}^{k-1} \binom{k-1}{i} =2^{k-1}$ equations for the coefficients $a_{l_1, \ldots, l_i}$. These give a linear homogeneous system of equations, and we want to prove it admits only the trivial solution. We transform the system as follows.
\begin{itemize}
\item[•] First step. We subtract the equality obtained in $0)$ from all the remaining equalities, and then divide by $\alpha$. Hence we are left with $2^{k-1}-1$ equalities (we can forget about the one from $0)$) and, in all of them, there will appear elements $a_{l_1, \ldots, l_i}$ such that at least one $l_j$'s is among $\{1, \dots, k-1 \}$, $i =1$, \dots, $k-1$.
\item[•] Second step. We consider the last one of the obtained equations (notice that we have only one evaluation from the k-1) substitution) and we use the following trick: we subtract from it all the equations from k-2), then add all the equations from k-3) and so on alternating addition and subtraction.  
\end{itemize}

We claim that the equation obtained is just 
\[
a_{1, \ldots, k-1} = 0.
\]
In fact, $a_{1, \ldots, k-1}$ appears, with coefficient 1, in all the $2^{k-1}-1$ equations obtained after the first step. Hence, after performing the second step, the coefficient of $a_{1, \ldots, k-1}$ is exactly
\[
\sum_{i=1}^{k-1} (-1)^{k-i-1} {k-1 \choose i} = \begin{cases}1 & \mbox{ if } k \mbox{ is odd}  \\
-1 & \mbox{ if } k \mbox{ is even} .  \end{cases}
\]
In order to prove the claim we have to show that the coefficients of all the remaining $a_{l_1, \ldots, l_i}$ are zero. First let us focus our attention to the coefficient of the $a_i$'s, $i= 1$, \dots, $k-1$. In this case the coefficient of $a_i$ is exactly
\[
\sum_{i=0}^{k-2} (-1)^{k-i-2} {k-2 \choose i} = 0.
\]
Now let us consider the coefficients of the $a_{ij}$'s. If just one of $i$ and $j$ is in $\{1,\ldots, k-1\}$, then the coefficient of $a_{ij}$ is exactly  
\[
\sum_{i=0}^{k-2} (-1)^{k-i-2} {k-2 \choose i} = 0.
\]
On the other hand, if both $i$, $j \in \{1,\ldots, k-1\}$, we have that the coefficient of $a_{ij}$ is equal to  
\[
\sum_{i=0}^{k-2} (-1)^{k-i-2} {k-2 \choose i} + \sum_{i=0}^{k-3} (-1)^{k-i-3} {k-3 \choose i} = 0.
\]
Now let us consider the coefficients of the $a_{ijh}$'s. In case $i$, $j$, $h \in \{1,\ldots, k-1\}$, we have that the coefficient of $a_{ijh}$ is equal to 
\[
\sum_{i=0}^{k-2} (-1)^{k-i-2} {k-2 \choose i} + \sum_{i=0}^{k-3} (-1)^{k-i-3} {k-3 \choose i} + \sum_{i=0}^{k-4} (-1)^{k-i-4} {k-4 \choose i} = 0.
\]
As before, we consider separately the cases when just $i$ or  just $i$ and $j$ belong to $\{1,\ldots, k-1\}$. But then we get just the first one or the first two sums above. In any case we get zero as a result and we are done. 

In a similar way one deals with all remaining coefficients and so the claim is proved. 

Using the same argument but changing the set of the $k-1$ variables accordingly we get that all the $a_{l_1, \ldots, l_{k-1}}$'s are zero.

Hence we kill all traces of length $k-1$. Then it is sufficient to consider, in the same way as above the elements $a_{j_1, \ldots, j_{k-2}}$ and so on. At the end of this process we prove that the polynomials in \eqref{non identities of Ck t alpha 0} are linearly independent. In conclusion
\[
\Id^{tr} \left( C_k^{t_{\alpha, 0}} \right) = T.
\]
The second part of the theorem follows immediately since the elements in \eqref{non identities of Ck t alpha 0} are exactly $
\displaystyle \sum_{i = 0}^{k-1} \binom{n}{i}$.
\end{proof}

\section{Varieties of polynomial growth}

In this section we study  varieties of algebras with trace whose codimension sequence is of polynomial growth. We start with the following definition.

\begin{Definition}
A variety $\mathcal{V}$ of algebras with trace is minimal of polynomial growth if $c^{tr}_n(\mathcal{V})\approx qn^k$, for some $k\geq 1$, $q>0$, and for every proper subvariety $\mathcal{U}\subsetneqq \mathcal{V}$ generated by a unitary finite dimensional algebra, we have that $c^{tr}_n(\mathcal{U})\approx q^\prime n^t$, with $t<k$.
\end{Definition}

The following result (see \cite[Theorem 30]{IoppoloKoshlukovLaMattina2021}) gives a characterization of the varieties of polynomial codimension growth. 

\begin{Theorem} \label{trace algebras with PG}
Let $A = A_1 \oplus \cdots \oplus A_m + J$ be a unitary finite dimensional  algebra with trace $\tr$ over a  field $F$ of characteristic zero. Then $c_n^{tr}(A)$ is polynomially bounded if and only if 
$\tr(J) = 0$, $A_i \cong F$, $i = 1$, \dots, $m$, and either $m=1$ or $\tr(A_i) = 0$ and $A_i J A_k = 0$, $1 \leq i,k \leq m$, $i \neq k$.
\end{Theorem}

Given two algebras with trace $A$ and $B$, we say that $A$ is $T^{tr}$-equivalent to $B$, and we write $A \sim_{T^{tr}} B$, in case $\Id^{tr}(A) = \Id^{tr}(B) $.

\begin{Theorem} \label{C_k minimal varieties}
For every $k \geq 2$, the trace algebra $C_k^{t_{\alpha,0}}$ generates a minimal variety of polynomial growth.
\end{Theorem}
\begin{proof}
We have seen in Theorem \ref{identities of Ck t alfa 0} that $c_n^{tr}(C_k^{t_{\alpha,0}}) \approx r n^{k-1}$ where $r=1/(k-1)!$. 
Now let $ A \in \V^{tr} (C_k^{t_{\alpha,0}}) $ be a unitary finite dimensional algebra with trace $t$.

If $\Tr(x)\equiv 0$ on $A$ then,  since $A$ is commutative and with zero trace, we get that $c_n^{tr}(A) = 1$ and we are done.

Now assume that $\Tr(x) \not \equiv 0$ on $A$. Since $A \in \V^{tr} (C_k^{t_{\alpha,0}})$ then its sequence of codimensions is polynomially bounded. 
By Theorem \ref{trace algebras with PG} we get that $A \sim_{T^{tr}} F + J$ and, since $g_6 \equiv 0$ is a trace identity for $A$,   $t(1_A) = \alpha$. 
If $J=0$ then $c_n^{tr}(A) =1$ and we have finished this case.
So we assume that $J^{q-1} \neq 0$ but $J^q = 0$, for some $q\geq 2$, and consider $q-1$ elements $j_1$, \dots,  $j_{q-1} \in J$ such that $j_1 \cdots j_{q-1} \neq 0$. 

We shall prove that $q$ must be less than or equal to $k$.

Suppose by contradiction that $q > k$. Since  $j_1 \cdots j_{q-1} \neq 0$ we obtain that $j_1 \cdots j_{k} \neq 0$. By evaluating the trace identity $g_7$ on $j_1$, \dots, $j_{k}$,  we get
\[
g_7 (j_1, \ldots, j_k) = \left( \Tr(j_1) - \alpha j_1 \right) \cdots  \left( \Tr(j_k) - \alpha j_k \right) = (-1)^k \alpha^k j_1 \cdots j_k \neq 0,
\]
a contradiction.

Now if $q<k$, as in the proof of Theorem \ref{identities of Ck t alfa 0}, it is easily seen that 
\[
g = \left( \Tr(x_1) - \alpha x_1 \right) \cdots \left( \Tr(x_q) - \alpha x_q \right) \equiv 0 
\] 
is a trace identity for $A$ and $c_n^{tr}(A) \leq cn^{q-1} $, $q-1 < k-1$, and we are done.

Therefore we have to consider the case $q = k$. 
We shall  prove that $A \sim_{T^{tr}} C_k^{t_{\alpha, 0}}$. Since $\Id^{tr}( C_k^{t_{\alpha,0}}) \subseteq \Id^{tr}(A) $ we already know  that the polynomials in \eqref{non identities of Ck t alpha 0} span $MT_n$, modulo $MT_n \cap \Id^{tr}(A)$. In order to complete the proof we need to show that they are linearly independent. We use the same technique employed in the proof of Theorem \ref{identities of Ck t alfa 0} with the following modification. We substitute $b^i$ (recall that $b=e_{12}+e_{23}+\cdots+e_{k-1,k}$) with $j_1 \cdots j_i$, $i\geq 1$. Thus we obtain that the polynomials in \eqref{non identities of Ck t alpha 0} are linearly independent and this completes the proof. 
\end{proof}
 
\begin{Lemma} \label{commutative pg alg are Ck t a 0}
Let $A = F + J$ be a commutative unitary algebra with trace $t$ such that $t(1_A) = \alpha$ and $\tr(J) = 0$. Then for some $k\geq 1$ we have
\[
A \sim_{T^{tr}} C_k^{t_{\alpha,0}}.
\]
\end{Lemma}
\begin{proof}
Let $k$ be the least integer such that $J^k=0$. 

If $k=1$ then $A=F=C_1^{t_{\alpha,0}}$ and we are done. 

Now assume $k>1$.
Since $A$ is commutative, then clearly $[x_1,x_2] \equiv 0$ on $A$. 
Moreover, since $J^k = 0$, the trace polynomial  
\[
\left( \Tr(x_1) - \alpha x_1 \right) \cdots \left( \Tr(x_k) - \alpha x_k \right) 
\]
is a trace identity for $A$.
Finally, it is immediate to see that $ \Tr(x_1) \Tr(x_2) - \alpha \Tr(x_1 x_2)$ vanishes on $A$. Hence
\[
\Id^{tr} \left( C_k^{t_{\alpha,0}} \right) \subseteq \Id^{tr}(A). 
\]
The opposite inclusion follows from the proof of the previous theorem.
\end{proof}

In \cite{IoppoloKoshlukovLaMattina2021} we characterized the varieties generated by finite dimensional algebras whose trace codimensions are of polynomial growth .

\begin{Theorem}[\cite{IoppoloKoshlukovLaMattina2021}] \label{characterization}
Let $A$ be a unitary finite dimensional  algebra with trace $\tr$ over a field $F$ of characteristic zero. Then the sequence $c_n^{tr}(A)$,  $n=1$, 2, \dots, is polynomially bounded if and only if $ D_2^{t_{\alpha, \beta}}$,  $D_2^{t_{\gamma,\gamma}}$, $D_2^{t_{\delta,0}}$, $C_2^{t_{\epsilon,1}}$, $UT_2(F) \notin \V^{tr}(A)$, for every choice of $\alpha$, $\beta$, $\gamma$, $\delta \in F \setminus \{ 0 \}$, $\alpha \ne \beta$, $\epsilon \in F$.
\end{Theorem}
Here  $UT_2$ denotes the algebra of $2\times 2$ upper triangular matrices over $F$ endowed with zero trace. 

As a consequence of the above result, it is not difficult to see that the algebras in the previous theorem  are the only finite dimensional algebras with trace generating varieties of almost polynomial growth, that is  
varieties with the property that their codimension sequences grow exponentially  but any proper subvariety has polynomial growth.

In the following theorem we classify all the subvarieties of the varieties of almost polynomial growth  generated by unitary finite dimensional  algebras with non-zero trace.

\begin{Theorem} \label{classification of the subvarieties}
Let $A$ be a unitary finite dimensional algebra with trace $\tr$ over a field of characteristic zero. If $A \in \V^{tr}(B)$ where $B \in \{ D_2^{t_{\alpha', \beta}}, D_2^{t_{\gamma,\gamma}}, D_2^{t_{\delta,0}}, C_2^{t_{\epsilon,1}}\}$, for every choice of $\alpha'$, $\beta$, $\gamma$, $\delta \in F \setminus \{ 0 \}$, $\alpha' \ne \beta$, $\epsilon \in F$ then either $A\sim_{T^{tr}} B$ or
$A \sim_{T^{tr}} C_k^{t_{\alpha, 0}}$, for some $k \geq 1$ and $\alpha = \tr(1_A)$.
\end{Theorem}
\begin{proof}
If $A \sim_{T^{tr}} B$ there is nothing to prove. Hence let us suppose that $A$ generates a proper subvariety of $\V^{tr}(B)$. Since $A \in \V^{tr}(B)$, then $A$ is commutative and $c_n^{tr}(A)$, $n=1$, 2, \dots,  is polynomially bounded. By Theorem \ref{trace algebras with PG} we have that $ A = A_1 \oplus \cdots \oplus A_m + J$, $\tr(J) = 0$, $A_i \cong F$ and either $m=1$ or $\tr(A_i) = 0$, and $A_i J A_k = 0$, $1 \leq i,k \leq m$, $i \neq k$.
In the last case $A$ is a commutative algebra with zero trace and so $A \sim_{T^{tr}} F= C_k^{t_{0, 0}}$ and we are done.  
In the first case $A=F+J$ and the result follows from Lemma \ref{commutative pg alg are Ck t a 0}.
\end{proof}

We observe that in \cite{LaMattina2007, LaMattina2008} it was obtained the classification of all the subvarieties (not only generated by finite dimensional algebras) of $\V^{tr}(UT_2)$.  Hence this result completes the classification of the subvarieties of the varieties of almost polynomial growth.

\begin{Corollary} \label{subvarieties of D2 t a a}
Let $A$ be a unitary finite dimensional algebra with trace $\tr$ over a field of characteristic zero. If $\V^{tr}(A)\subsetneq \V^{tr}(D_2^{t_{\alpha, \alpha}})$, $\alpha \neq 0$, then either 
$A \sim_{T^{tr}}  C_1^{t_{\beta, 0}}$ with $\beta \in \{0, \alpha, 2\alpha \}$ or $A \sim_{T^{tr}} C_2^{t_{2 \alpha, 0}}$.
\end{Corollary}
\begin{proof}
By Theorem \ref{classification of the subvarieties} we know that 
$A \sim_{T^{tr}} C_k^{t_{\beta, 0}}$, for some $k \geq 1$ and $\beta = \tr(1_A)$.
If $\beta=0$ then $A\sim_{T^{tr}} C_1^{t_{0, 0}}$ and we are done.
So we assume that $\beta$ is different from zero. 

Since $A \in \V^{tr}(D_2^{t_{\alpha, \alpha}})$, we have that $A$ satisfies the trace identity
\[
f(x_1, x_2) = \Tr(x_1)\Tr(x_2) + \alpha^2 x_1x_2 + \alpha^2 x_2 x_1 - \alpha \Tr(x_1)x_2 - \alpha \Tr(x_2)x_1 - \alpha \Tr(x_1 x_2).
\] 
Hence, by considering the evaluation $x_1 = x_2 = 1_A$ we get that
\[
f(1_A, 1_A) = \left( \tr(1_A) \right)^2 - 3 \alpha \tr(1_A) + 2 \alpha^2 = 0
\]
and, so, either $\tr(1_A) = \alpha$ or $\tr(1_A) = 2 \alpha$.

Hence if $k=1$ we get that either $A \sim_{T^{tr}}  C_1^{t_{\alpha, 0}}$ or $A \sim_{T^{tr}}  C_1^{t_{2\alpha, 0}}$.

Now assume that $k>1$ and,
without loss of generality, that $A= C_k^{t_{\beta, 0}}$. We shall prove that $\tr(1_A) = 2\alpha$ and $k=2$.
Suppose that $\tr(1_A) = \alpha$. Then if $0\neq j\in J$ we get  $f(1_A, 1_A+j) = \alpha^2 j = 0$, a contradiction.

Hence $\tr(1_A) = 2 \alpha$. Moreover, for every $j$, $k \in J$ we have 
\[
f(j,k) = 2 \alpha^2 jk = 0.
\]
It follows that $J^2 = 0$ and therefore $k=2$. This completes the proof.
\end{proof}

\end{document}